\documentclass[a4paper]{amsart}

\usepackage[english]{babel}

\usepackage{amssymb}
\usepackage{amsthm}
\usepackage{mathtools}
\usepackage{dsfont}

\usepackage{microtype}

\usepackage[a4paper, inner=3cm,outer=3cm,top=3.3cm,bottom=3cm]{geometry}

\usepackage{xspace}

\usepackage[autostyle]{csquotes}

\usepackage{tikz}
\usetikzlibrary{matrix,arrows,patterns,decorations.pathreplacing}
\usepackage{pgfplots}

\usepackage[colorinlistoftodos, bordercolor=orange, backgroundcolor=orange!20,
linecolor=orange, textsize=scriptsize,textwidth=25mm]{todonotes}
\usepackage[colorlinks, bookmarksdepth=3]{hyperref}

\makeatletter
\@ifpackageloaded{hyperref}{%
\DeclareRobustCommand{\SkipTocEntry}[5]{}}{%
\DeclareRobustCommand{\SkipTocEntry}[4]{}}
\makeatother

\theoremstyle{plain}
\newtheorem{thm}{Theorem}[section]
\newtheorem{prop}[thm]{Proposition}

\newtheorem{lem}[thm]{Lemma}
\newtheorem{conj}[thm]{Conjecture}

\theoremstyle{definition}
\newtheorem{defn}[thm]{Definition}
\newtheorem{ex}[thm]{Example}
\newtheorem{exes}[thm]{Examples}
\newtheorem{rem}[thm]{Remark}

\newcommand{\lin}{{\mathrm{lin}}}
\newcommand{\aff}{{\mathrm{aff}}}

\DeclareMathOperator{\conv}{conv}
\DeclareMathOperator{\codeg}{codeg}

\newcommand{\R}{{\mathds{R}}}
\newcommand{\Z}{{\mathds{Z}}}

\newcommand{\rleft}{\mathopen{}\mathclose\bgroup\left}
\newcommand{\rright}{\aftergroup\egroup\right}

\newcommand{\pro}[2]{\langle #1, #2 \rangle}
\title[Proof of a conjecture of Batyrev and Juny on Gorenstein polytopes]{Proof of a conjecture of Batyrev and Juny \\ on Gorenstein polytopes}

\author[B.\,Nill]{Benjamin Nill}
\address{Fakult\"at f\"ur Mathematik\\OVGU Magdeburg\\Universit\"atsplatz 2\\39106 Magdeburg\\Germany}
\email{benjamin.nill@ovgu.de}


\begin{document}

\selectlanguage{english}

\begin{abstract}
A $d$-dimensional lattice polytope $P$ is Gorenstein if it has a multiple $r P$ that is a reflexive polytope up to translation by a lattice vector. The difference $d+1-r$ is called the degree of $P$. 
We show that a Gorenstein polytope is a lattice pyramid if its dimension is at least three times its degree. 
This was previously conjectured by Batyrev and Juny. We also present a refined conjecture and prove it for IDP Gorenstein polytopes.
\end{abstract}

\dedicatory{This paper is dedicated to Victor V. Batyrev on the occasion of his 60th birthday.}

\maketitle{}

\section{Notation and basic notions}
\label{sec:notation}

This section may be skipped on a first reading.

\subsection{Lattice polytopes and lattice pyramids}

For a subset $A \subseteq \R^d$ we denote by $\conv(A)$, $\lin(A)$, $\aff(A)$ its convex, linear, affine hull respectively. Given $A,B \subseteq \R^d$, their {\em Minkowski sum} is $A+B := \{a + b \,:\, a \in A,\, b \in B\}$. Let us recall that a {\em lattice polytope} $P \subset \R^d$ is the non-empty convex hull of finitely many elements of the lattice $\Z^d$. Two lattice polytopes are {\em isomorphic} or {\em unimodularly equivalent} if they are mapped to each other by an affine lattice-preserving transformation. The $d$-dimensional {\em standard simplex} is given as $\Delta_d := \conv(0,e_1,\ldots,e_d)$, where $0$ denotes the origin of $\R^d$ and $e_1, \ldots, e_d$ the standard basis vectors. 
A lattice polytope is called {\em hollow} if it does not contain a lattice point in its relative interior. A lattice point is by convention not hollow. We denote a lattice polytope $P \subset \R^d$ as a {\em lattice pyramid} if $P$ is a lattice point or if there is a lattice polytope $P' \subset \R^{d-1}$ such that $P \cong \conv(P' \times \{0\}, \{0\} \times \{1\}) \subset \R^{d-1} \times \R$. For instance, $\Delta_d$ is a lattice pyramid. Note that in this paper, by definition, lattice points are considered to be lattice pyramids.

\subsection{Lattice distance}

Given the facet $F$ of a full-dimensional lattice polytope $P \subset \R^d$ and a point $x \in \R^d$, the {\em lattice distance of $F$ from $x$} is defined as the difference $\pro{u}{F}-\pro{u}{x}$, where $u$ is the unique primitive integral linear functional on $\R^d$ that evaluates constantly on $F$. We say, the lattice distance {\em is given by $u$}. For a lower-dimensional lattice polytope $P \subset \R^d$ the lattice distance of a facet $F$ from $x \in \aff(P)$ is defined with respect to the affine lattice of lattice points in its affine hull. In this case, $u$ may be taken as a primitive integral affine functional on $\aff(P)$, where primitive is equivalent to the existence of two lattice points whose evaluations with $u$ differ only by~$1$. Hence, a facet $F$ of a lattice polytope $P$ has {\em lattice distance one} from $x \in \aff(P)$ if and only if there exists some integral affine functional $u$ on $\aff(P)$ that evaluates constantly on $F$ and satisfies $\pro{u}{F}-\pro{u}{x}=1$.

\subsection{Degree and codegree}

The {\em codegree} of a $d$-dimensional lattice polytope $P$ is the smallest positive integer $r$ such that $rP$ is not hollow. The difference $s := d+1-r$ is called the {\em degree} of $P$, we refer to \cite{BN07,BR15}. The degree of a lattice polytope should be seen as a measure for the complexity of $P$ that is more refined than the dimension $d$. One has $0 \le s \le d$, with $s=0$ if and only if $P \cong \Delta_d$, and $s=d$ if and only if $P$ is not hollow. Its importance stems from its interpretation in Ehrhart theory as the degree of the enumerator polynomial (the so-called {\em $h^*$-polynomial}) of the rational generating function of its Ehrhart polynomial. 

\subsection{Reflexive and Gorenstein polytopes}

A $d$-dimensional lattice polytope $P \subset \R^d$ is {\em reflexive} if there is an interior lattice point $x$ such that every facet has lattice distance one from $x$. In this case, $x$ is the unique interior lattice point of $P$. Note that throughout this paper, reflexive polytopes are {\em not} assumed to have the origin $0$ as the unique interior lattice point (as is usually required in the literature). We say, a lattice polytope $P$ is {\em Gorenstein} if there is some $r \in \Z_{\ge 1}$ such that $r P$ is reflexive. In this case, $r$ is uniquely determined and often called the {\em index} of $P$. It is equal to the codegree of $P$. A lattice polytope is Gorenstein if and only if its $h^*$-polynomial is palindromic \cite{Hibi92}. Let us just shortly point out that reflexive and Gorenstein polytopes have a beautiful duality property and turn up naturally in toric geometry and combinatorial commutative algebra. The original motivation for their intensive studies over the previous decades stems from their significance in constructing mirror-symmetric Calabi-Yau complete intersections in Gorenstein toric Fano varieties and the computation of their stringy Hodge numbers. We refer to \cite{Bat94, BB94,BB96,BN08} for more on this. 

\section{The main results}

\subsection{The Batyrev-Juny conjecture}

The classification of lattice polytopes of given degree $s$ is an active subject of study in the area of Ehrhart theory and lattice polytopes. We refer to the introduction in \cite{HN21} for pointers to the literature. Note that for fixed degree the dimension of a lattice polytope can still be arbitrarily large. Lattice polytopes of degree $s\le1$ were completely classified by Batyrev and the author \cite{BN08}. In \cite{BJ10} Batyrev and Juny managed to give a complete list of all Gorenstein polytopes of degree $\le 2$, thereby providing a complete classification of Gorenstein toric Del Pezzo varieties in arbitrary dimension. Let us explain why theoretically such a result is possible for Gorenstein polytopes $P$ of {\em any} fixed degree $s$. For this, recall that lattice pyramids over Gorenstein polytopes are again Gorenstein polytopes of the same degree. Now, it follows from \cite[Cor.~3.2]{HNP09} and \cite{Bat06} that there is a function $d(s)$ such that if $\dim(P) \ge d(s)$, then $P$ is a lattice pyramid. As in each dimension there are only a finite number of reflexive polytopes \cite{Bat94} (up to isomorphisms), and thus Gorenstein polytopes, this shows that there is only a {\em finite} number of non-lattice-pyramid Gorenstein polytopes of fixed degree $s$ (but arbitrary dimension). Therefore, finding a sharp value of $d(s)$ is important for future classification results of Gorenstein polytopes and their associated toric varieties. However, simply going through the proofs in \cite{HNP09} and \cite{Bat06} would yield a bound that is at least doubly-exponential in $s$. Based upon the existing classification results of Gorenstein polytopes for $s\le 2$,  Batyrev and Juny conjectured in \cite[Conjecture~0.1]{BJ10} that $d(s)=3s$. Here, we prove this conjecture.

\begin{thm}Let $P$ be a Gorenstein polytope of dimension $d$ and degree $s$. If $d \ge 3s$, then $P$ is a lattice pyramid.
\label{BJ-conj}
\end{thm}

The proof is given in Section~3.

Let us discuss why Theorem~\ref{BJ-conj} is sharp. For this, we need the important notion of a Cayley polytope that is also essential in the proof of the conjecture.

\begin{defn}
Let $P_1, \dotsc , P_l \subset \R^n$ be lattice polytopes. We define the \emph{Cayley polytope} \linebreak $P_1 * \cdots * P_l \subset \R^{n+l}$ \emph{associated to} 
$P_1, \ldots, P_l$ as the convex hull of 
$$
(P_1 \times \{e_1\}) \cup \dots \cup (P_l \times \{e_l\})
$$
In this case one has $\dim(P_1 * \dots * P_l) = \dim(P_1 + \dots + P_l) + l-1$. 

Throughout, we will refer to a lattice polytope $P \subset \R^d$ as a \emph{Cayley polytope} if $P$ is a lattice point or $P \cong P_1* \cdots * P_l$ for some lattice polytopes $P_1, \dotsc, P_l$ with $l > 1$. Note that in this paper, by definition, lattice points are considered to be Cayley polytopes. 
\end{defn}

Clearly, lattice pyramids are a special case of Cayley polytopes. The following well-known result, e.g., \cite[Theorem~2.6]{BN08}, shows how to construct Gorenstein polytopes as Cayley polytopes of special families of lattice polytopes. It will be generalized in Proposition~\ref{prop:cayley}.

\begin{prop}
Let $P_1, \ldots, P_l$ be lattice polytopes, and let $P :=P_1 * \cdots * P_l$ be their Cayley polytope. Then $P$ is a Gorenstein polytope of codegree $l$ if and only if $P_1 + \cdots + P_l$ is reflexive.
\label{cayley-old}
\end{prop}

This allows us to see why Theorem~\ref{BJ-conj} is sharp. Take for $s\ge1$
\[P_1 := \conv(0,e_1), \ldots, P_s := \conv(0,e_s), P_{s+1} := \conv(0,-e_1), \ldots, P_{2s} := \conv(0,-e_s).\]
Here, $P_1 + \cdots + P_{2s} = [-1,1]^s$ is reflexive. Hence, Proposition~\ref{cayley-old} implies that the Cayley polytope $P_1 * \cdots * P_{2s}$ is a Gorenstein polytope of dimension $3s-1$ and degree $s$ which is not a lattice pyramid.

It has also been conjectured by Batyrev and Juny in \cite[Conjecture~0.1]{BJ10} that this is the only such extreme example. This part of their conjecture is still open.

\subsection{The Batyrev-Juny conjecture refined}

The following result, \cite[Thm.~3.1]{HNP09}, is analogous to Theorem~\ref{BJ-conj} and shows that while Gorenstein polytopes with $2s \le d \le 3s-1$ may not be lattice pyramids, 
they still decompose as Cayley polytopes. 

\begin{thm}
Let $P$ be a Gorenstein polytope of dimension $d$ and degree $s$. If $d \ge 2s$, then $P$ is a Cayley polytope of lattice polytopes in $\R^n$ with $n \le 2s-1$.
\label{thm:cayley}
\end{thm}

The bound is sharp. Recall that $\Delta_d$ denotes the $d$-dimensional standard simplex. We see that $2 \Delta_{2d-1}$ is a Gorenstein simplex of dimension $2d-1$ and degree $d$ that is not a Cayley polytope.

Based upon the complete classification of Gorenstein polytopes of degree $\le 2$ it is reasonable to conjecture that the conclusion in Theorem~\ref{thm:cayley} should be considerably stronger. For this, we need the notion of Cayley joins and reducible Gorenstein polytopes introduced in \cite{NS13}. 

\begin{defn}
Let $F, G \subset \R^n$ be lattice polytopes. Then $P$ is called \emph{Cayley join} of $F$ and $G$ if $P \cong F*G$ and $\dim(F) + \dim(G) = \dim(P)-1$.
\end{defn}

By taking the dual Gorenstein polytope into account, it is possible to tighten the definition of a Cayley join for Gorenstein polytopes further. Instead of giving here the original definition of \emph{reducible} (and otherwise {\em irreducible}) Gorenstein polytopes, \cite[Def.~4.13]{NS13}, the following formulation is more suitable for our purposes. The facts that these two definitions agree and that the factors are Gorenstein are essentially contained in the proofs of \cite[Thm.~5.8]{NS13} and \cite[Thm.~4.12]{NS13}; details will be given in an upcoming paper \cite{Upcoming}.

\begin{defn}
Let $P$ be a Gorenstein polytope. We say $P$ {\em is reducible with factors} $F$ and $G$ if $P$ is a Cayley join of $F$ and $G$ and $\codeg(F) + \codeg(G) = \codeg(P)$ (equivalently, $\deg(F) + \deg(G)=\deg(P))$. 
In this case, $F$ and $G$ are also Gorenstein polytopes.
\label{next1}
\end{defn}

Note that if a Gorenstein polytope is a lattice pyramid, then it is also reducible (by convention, we consider a lattice point as reducible). We are now ready to formulate the following strengthening of Theorem~\ref{thm:cayley} as a (quite bold) conjecture.

\begin{conj}
Let $P$ be a Gorenstein polytope of dimension $d$ and degree $s$. If $d > 2s$, then $P$ is reducible.
\label{conj-new}
\end{conj}

It has an affirmative answer for $s \le 2$ by the complete classification results in \cite{BN07,BJ10}. Let us also note that it holds for Gorenstein simplices: it follows from \cite[Cor.~3.10(1)]{DHNP13} that Gorenstein simplices are lattice pyramids if $d \ge 2s$.

\begin{rem}
The bound in Conjecture~\ref{conj-new} would be sharp: Let $P := \Delta_s \times \Delta_s$ for $s>0$. Note that $P$ is the Cayley polytope of $s+1$ copies of $\Delta_s$. Moreover, $P$ is a Gorenstein polytope of dimension $2s$ and degree $s$. It is left to the reader to check that it is an irreducible Gorenstein polytope, in fact, it is not even a (Cayley) join. 
As the classification of Gorenstein polytopes of degree $2$ in dimension $4$ shows \cite{BJ10}, this extreme example is not unique.
\end{rem}

\begin{ex} Let us sum up above results and conjectures for Gorenstein polytopes $P$ of degree $s=2$ as classified in \cite{BJ10}. For $d=4$, any $P$ is a Cayley polytope, compare Theorem~\ref{thm:cayley}. For $d=5$, any $P$ is reducible, in accordance with Conjecture~\ref{conj-new}. In fact, there is only one non-lattice-pyramid $5$-dimensional Gorenstein polytope of degree $2$, namely $\conv(0,e_1,e_2,e_1+e_2) * \conv(0,e_3,e_4,e_3+e_4)$. From $d \ge 6$ onwards, there is no non-lattice-pyramid $P$, as stated by Theorem~\ref{BJ-conj}.
\end{ex}

We can show that Conjecture~\ref{conj-new} holds for a large class of Gorenstein polytopes. Let us recall that a lattice polytope $P$ is \emph{IDP} if for each $k \in \Z_{\ge 2}$ any lattice point in $kP$ is a sum of $k$ lattice points in $P$.

\begin{prop}
Let $P$ be an IDP Gorenstein polytope of dimension $d$ and degree $s$. If $d > 2s$, then $P$ is reducible.
\label{IDP}
\end{prop}

The proof is given at the end of Section~3. 

\begin{rem}
Proposition~\ref{IDP} shows that in order to construct {\em all} IDP Gorenstein polytopes of given degree $s$, it suffices to know the finitely many IDP Gorenstein polytopes of degree at most~$s$ {\em in dimension at most $2s$}. This follows from the following observations. First, if $P$ is a reducible IDP Gorenstein polytope, then also its factors are IDP Gorenstein polytopes. Second, if a Cayley join of $F$ and $G$ is IDP, then it is automatically a {\em free join} (also called {\em $\Z$-join}), i.e., it is isomorphic to the convex hull of $F \times \{0\} \times \{0\}$ and $\{0\} \times G \times \{1\}$. Third, as $h^*$-polynomials of free joins multiply (e.g., \cite[Lemma~1.3]{HenkTagami}), the free join of two Gorenstein polytopes of degrees $s_1$ and $s_2$ is a Gorenstein polytope of degree $s_1+s_2$. 
\end{rem}

Finally, let us observe that Conjecture~\ref{conj-new} is indeed a refinement of the Batyrev-Juny conjecture proven above. For this, we need the following statement that is contained in the master thesis of Michael \cite{Mic20}; a detailed proof will also be given in the upcoming paper \cite{Upcoming}.

\begin{lem}
Let $P$ be a Gorenstein polytope that is reducible with factors $F$ and $G$. If $F$ is a lattice pyramid, then $P$ is a lattice pyramid.
\label{next2}
\end{lem}

This allows us to deduce the following implication.

\begin{prop}
Conjecture~\ref{conj-new} implies Theorem~\ref{BJ-conj}.
\label{implication}
\end{prop}

\begin{proof}

By induction on the dimension we may assume that Theorem~\ref{BJ-conj} holds for Gorenstein polytopes of dimension $< d$ (the case $d \le 1$ being evident). Let $P$ be a Gorenstein polytope of dimension $d \ge 2$ and degree $s$. Let $d \ge 3s$, and $P$ be reducible with factors $F$ and $G$. By Lemma~\ref{next2} we may assume that $F$ and $G$ are not lattice pyramids (in particular their degrees are positive). Hence, by induction hypothesis we have $\dim(F) \le 3 \deg(F) - 1$ and $\dim(G) \le 3 \deg(G)-1$. This implies $\dim(P) = \dim(F) + \dim(G) + 1 \le 3 (\deg(F) + \deg(G)) - 1 = 3 s - 1$, a contradiction.
\end{proof}

\section{Proofs of Theorem~\ref{BJ-conj}  and Proposition~\ref{IDP}}

Both proofs are quite short, however, they need as preparation some generalizations and strenghtenings of results in \cite{BN08}. Let us first generalize Proposition~\ref{cayley-old}. 	

\begin{prop}
Let $P_1, \ldots, P_l \subset \R^n$ be lattice polytopes such that $\dim(P_1+\cdots+P_l) = n$, and let $P :=P_1 * \cdots * P_l$ be their Cayley polytope. Then $P$ is a Gorenstein polytope of codegree $r$ if and only if there exists $(k_1, \ldots, k_l) \in \Z_{\ge 1}^l$ with $k_1 + \cdots + k_l = r$ such that 

\begin{enumerate}
\item $k_1 P_1 + \cdots + k_l P_l$ is reflexive and
\item $\dim(\sum_{j \not=i} P_j) < n$ if $k_i \ge 2$ for $i \in \{1, \ldots, l\}$.
\end{enumerate}
\label{prop:cayley}
\end{prop}

\begin{proof}
Note that $\dim(P) = n + l -1$. Let $P$ be a Gorenstein polytope of codegree $r$. Hence, $rP$ is reflexive with respect to a unique interior lattice point $x := (x',(k_1, \ldots, k_l)) \in \Z^n \times \Z_{\ge 1}^l$ with $k_1 + \cdots + k_l=r$. Let $A$ denote the affine subspace $\R^n \times \{(k_1, \ldots, k_l)\}$. By the well-known Cayley trick (e.g., \cite[Sect.~9.2]{Triangulations}) we have
\begin{equation}
(r P) \cap  A = (k_1 P_1 + \cdots + k_l P_l) \times \{(k_1, \ldots, k_l)\}.
\tag{*}\label{trick}
\end{equation} Let $F'$ be a facet of $k_1 P_1 + \cdots + k_l P_l$. Then $F'$ is given as the intersection of a facet $F$ of $rP$ with $A$. As $r P$ is reflexive, $F$ has lattice distance one from $x$ given by an integral affine functional $u$. Restricting $u$ to $A$, we see that $F'$ also has lattice distance one from $x'$. Hence, (1) follows. Now, assume there exists $i \in \{1, \ldots, l\}$ such that $k_i \ge 2$ and $\dim(\sum_{j \not=i} P_j) = n$. We may assume $i=l$. Hence, $P_1 * \cdots * P_{l-1}$ is a facet of $P$. Note that $e^*_l$ restricts to a primitive integer affine functional on $\aff(rP)$ that evaluates to $0$ on the facet $r(P_1 * \cdots * P_{l-1})$ of $rP$ and to $k_l$ on 
$x$, thus, $r(P_1 * \cdots * P_{l-1})$ has lattice distance $k_l \ge 2$ from $x$ which is a contradiction to $rP$ being reflexive. This proves (2).

Conversely, let $(k_1, \ldots, k_l) \in \Z_{\ge 1}^l$ with $k_1 + \cdots + k_l = r$ satisfying (1) and (2). Let $x' \in \Z^n$ be the unique interior lattice point of $k_1 P_1 + \cdots + k_l P_l$, and $x := (x',(k_1, \ldots, k_l))$. Again by equation~\eqref{trick}, $x$ is an interior lattice point of $rP$. Let $F$ be a facet of $P$. There are two cases:

First, assume that there exists some $i \in \{1, \ldots, l\}$ such that $F$ does not contain a vertex in $P_i \times \{e_i\}$. We may assume $i=l$. Hence, $P_1 * \cdots * P_{l-1} = F$ is a facet of $P$, thus, has dimension $n+l-2$. Hence, $\dim(P_1 + \cdots + P_{l-1}) = n$. Condition (2) implies $k_l=1$. As above, it holds that $r(P_1 * \cdots * P_{l-1})$ has lattice distance $k_l$ from $x$, so $r F$ has lattice distance one from $x$.

Otherwise, we see that $F = F_1 * \cdots * F_l$ with non-empty faces $F_i$ of $P_i$ for $i=1, \ldots, l$. 
Let $u=(u',u'') \in (\Z^n)^* \times (\Z^l)^*$ be the unique primitive integral linear functional that evaluates to $0$ on $F$ and is nonnegative on $P$. 
As is well-known from the Cayley trick, we have $r F \cap (\R^n \times \{(k_1, \ldots, k_l)\}) = F' \times \{(k_1, \ldots, k_l)\}$ for some facet $F'$ of $k_1 P_1 + \cdots + k_l P_l$ (more precisely, $F'=k_1 F_1 + \cdots + k_l F_l$). Thus, $\pro{u}{F' \times \{(k_1, \ldots, k_l)\}}=0$. Let $b := \pro{u''}{(k_1, \ldots, k_l)} \in \Z$. Hence, $u'$, the restriction of $u$ onto $\Z^n$, evaluates on $F'$ to $-b$ and on $x'$ to $-b+k$ (for some $k \in \Z_{\ge 1}$). By condition (1) there exists a unique primitive integral linear functional $g \in (\Z^n)^*$ that evaluates on $F'$ to $c$ (for some $c \in \Z$) and on $x'$ to $c+1$. Therefore, there exists some $q \in \Z_{\ge 1}$ such that $u'=q g$. Hence, 
$qc=-b$ and $q(c+1)=-b+k$, which implies $q=k$. So we have $u'=k g$ and $-b=kc$. We write $u''=\sum_{i=1}^l c_i e^*_i$ where $c_1, \ldots, c_l \in \Z$. For $i=1, \ldots, l$ let $f_i$ be a vertex of $F_i$. Then $0 = \pro{u}{(f_i,e_i)} = \pro{u'}{f_i} + \pro{u''}{e_i} = k \pro{g}{f_i} + c_i$, thus, $k$ divides $c_1, \ldots, c_l$. As $u=(k g,u'')$ is a primitive lattice point, this implies $k=1$. Notice that, as $u'=g$, $c=-b$. From this, we get $\pro{u}{x} = \pro{u'}{x'} + \pro{u''}{(k_1, \ldots, k_l)} = -b+1+b=1$. Hence, the facet $rF$ has lattice distance one from $x$, as desired.

This shows that $rP$ is a reflexive polytope, hence, $P$ is a Gorenstein polytope of codegree $r$.

\end{proof}

\begin{exes}
Let us illustrate this result for $n=2$.

\begin{enumerate}
\item Let $P_1 := [0,1] \times \{0\}$ and $P_2 := \{0\} \times [0,1]$. Then $2 P_1 + 2 P_2$ is reflexive. Hence, $P_1 * P_2$ is a Gorenstein polytope of codegree $4$. In fact, $P_1*P_2 \cong \Delta_3$.
\item Let $P_1 := \Delta_2$ and $P_2 := \{0\} \times [0,1]$. Then $2 P_1 + P_2$ is reflexive. Here, $P_1 * P_2$ is a three-dimensional Gorenstein polytope of codegree $3$.
\item Let $P_1 := P_2 := \Delta_2$. Note that while $P_1 + 2P_2$ is reflexive, $P_1 * P_2$ is a lattice polytope of codegree $3$ but not a Gorenstein polytope (condition (2) is violated).
\end{enumerate}
\end{exes}

Let us recall the following notion from \cite{BB96} that has been intensively discussed in \cite{BN08}.

\begin{defn}
Let $P_1, \dotsc , P_l \subset \R^n$ be lattice polytopes. We say $P_1, \ldots, P_l$ is a {\em nef-partition} if $P_1 + \cdots + P_l$ is a reflexive polytope with unique interior lattice point $x$ and there exist lattice points $p_1 \in P_1, \ldots, p_l \in P_l$ such that $p_1 + \cdots + p_l = x$. In the case that $p_1=\cdots=p_l=x=0$, the nef-partition is called {\em centered}.
\end{defn}

The first part of the next result generalizes \cite[Prop.~6.16]{BN08}, where it was proven for centered nef-partitions. Crucial new ingredient is the nonnegativity of the mixed degree, introduced in \cite{Nil20}.

\begin{lem}
Let $P_1, \ldots, P_l \subset \R^n$ be lattice polytopes with $\dim(P_i) \ge 1$ for $i=1, \ldots, l$ such that $k_1 P_1 + \cdots + k_l P_l$ is reflexive (up to a translation) for $k_1, \ldots, k_l \in \Z_{\ge 1}$. Then we have $k_1 + \cdots + k_l \le 2n$. Moreover, if $k_1 + \cdots + k_l \ge n+2$ and, additionally, $\dim(\sum_{j \not=i} P_j) < n$ if $k_i \ge 2$ for $i \in \{1, \ldots, l\}$, then the Cayley polytope $P_1* \cdots * P_l$ is a reducible Gorenstein polytope.
\label{lem:mink-bound}
\end{lem}

\begin{proof}

Let us prove the first statement by induction on the dimension $n$. Clearly, it holds for $n \le 1$, so we may assume $n \ge 2$. 
Let $r := k_1 + \cdots + k_l$. We denote by $Q_1, \ldots, Q_r$ the family $P_1, \ldots, P_1, \ldots, P_l, \ldots, P_l$ where for $i=1, \ldots, l$ each $P_i$ occurs with multiplicity $k_i$. 
In the following, we use for the Minkowski sum the notation $Q_I := \sum_{i \in I} Q_i$ for $\emptyset \not= I \subseteq [r]$. We also set $[r] := \{1, \ldots, r\}$. 

Let us first assume that $Q_I$ is hollow for all $\emptyset \not= I \subsetneq [r]$. In the notation of [Nil20] this means that the mixed codegree of the family $Q_1, \ldots, Q_r$ is at least $r$. On the other hand, the nonnegativity of the mixed degree \cite[Prop.~4]{Nil20} implies that the mixed codegree is at most $n+1$. Hence, $r \le n+1 \le 2n$. 

Otherwise, there exists a subset $\emptyset \not= I \subsetneq [r]$ such that $Q_I$ is not hollow. Let $J := [r] \setminus I$. As $Q_{[r]} = Q_I + Q_J = k_1 P_1 + \cdots + k_l P_l$ is reflexive, it follows from \cite[Prop.~6.13]{BN08} that $Q_I$ and $Q_J$ form a nef-partition, and $Q_I$ and $Q_J$ are reflexive in their ambient lattices with respect to their relative interior lattice points $q \in Q_I$ and $q' \in Q_J$. Applying \cite[Prop.~6.11]{BN08} to the (so-called reducible) centered nef-partition $Q_I - q$ and $Q_J-q'$, we see that 
\[\lin(Q_I-q) \oplus_\R \lin(Q_J-q') = \R^n.\] In particular, we have $n = \dim(Q_I) + \dim(Q_J)$. As by our assumption the dimensions of $Q_I$ and $Q_J$ are at least one, we see that they are strictly smaller than $n$. Therefore, induction implies that $r = |I| + |J| \le 2 \dim(Q_I) + 2 \dim(Q_J) =  2 n$.

For the additional statement, let $r \ge n+2$, in particular, $n \ge 2$. We follow the above argument and deduce in the previous notation that there exists a subset $\emptyset \not= I \subsetneq [r]$, $J := [r] \setminus I$, such that $Q_I-q$ and $Q_J-q'$ is a centered nef-partition. 

Assume first that the family $\{Q_i\}_{i \in I}$ contains each $P_1, \ldots, P_l$ at least once. In this case, $\dim(Q_I) = \dim(P_1 + \cdots + P_l) = n$, thus, $\dim(Q_J)=0$, a contradiction. By the same argument applied to $\{Q_j\}_{j \in J}$, we may assume without loss of generality that $1 \in I$ and $r \in J$ such that $Q_1=P_1$ is not in $\{Q_j\}_{j \in J}$ and $Q_r=P_l$ is not in $\{Q_i\}_{i \in I}$. Denote the family $Q_1-q, Q_2, \ldots, Q_{r-1},Q_r-q'$ by $R_1, \ldots, R_r$. Note that $R_I = Q_I-q$ and $R_J = Q_J - q'$. Let us introduce coordinates $x=(x_1,x_2) \in \R^n$ for the $\R$-splitting $\lin(R_I) \oplus_\R \lin(R_J) = \R^n$.

Assume next that there exists some $t \in \{2, \ldots, r-1\}$ such that $Q_t$ appears in $\{Q_i\}_{i \in I}$ and $\{Q_j\}_{j \in J}$. Without loss of generality let $t \in I$ and $Q_t=P_h$ for $h \in \{1, \ldots, l\}\setminus\{1,l\}$. 
Then there exists $t' \in J \cap \{2, \ldots, r-1\}$ with $R_t=Q_t = P_h=Q_{t'}=R_{t'}$. Choose $v \in R_t=R_{t'}$, $\{r_i\}_{i \in I\setminus\{t\}} \in \{R_i\}_{i \in I\setminus\{t\}}$ and $\{r'_j\}_{j \in J\setminus\{t'\}} \in \{R_j\}_{j \in J\setminus\{t'\}}$. We have $\left(v+\sum_{i \in I\setminus\{t\}} r_i\right)_2=0$ and $\left(v+\sum_{j \in J\setminus\{t'\}} r'_j\right)_1=0$. In particular, $v_1 = - \left(\sum_{j \in J\setminus\{t'\}} r'_j\right)_1$ and $v_2 = - \left(\sum_{i \in I\setminus\{t\}} r_i\right)_2$. (Note that we do not claim that $(v_1,v_2)=(0,0)$.) As this holds for any choice of $v=(v_1,v_2) \in R_t$, this implies $\dim(R_t)=0$, a contradiction. 

This shows that for $i =1, \ldots, l$ the lattice polytope $P_i$ appears $k_i$-times in $\{Q_i\}_{i \in I}$ and not in $\{Q_j\}_{j \in J}$, or vice versa. Thus we may assume without loss of generality that there is $c \in \{1, \ldots, l-1\}$ such that $k_1 P_1 +  \cdots + k_c P_c = Q_I$ and $k_{c+1} P_{c+1} + \cdots + k_l P_l=Q_J$ are reflexive in their ambient lattices. 

By our assumptions, Proposition~\ref{prop:cayley} implies that $P := P_1 * \cdots * P_l$ is a Gorenstein polytope of codegree $r$. Let us consider the faces $F := P_1 * \cdots * P_c$ and $G := P_{c+1} * \cdots * P_l$ of $P$. As we have $\dim(F) = \dim(Q_I) + c - 1$, 
$\dim(G) = \dim(Q_J) + l-c-1$ and $\dim(P) = n + l - 1$, we get $\dim(F) + \dim(G)  = \dim(P) - 1$. Hence, $P$ is a Cayley join of $F$ and $G$.

If $k_1 \ge 2$ and $\dim(\sum_{i=2}^c P_i) = \dim(P_1 + \cdots + P_c)$, then also $\dim(\sum_{i=2}^l P_i) = \dim(P_1 + \cdots + P_l)$, a contradiction. Hence, applying 
Proposition~\ref{prop:cayley} we get that $F$ is a Gorenstein polytope of codegree $k_1 + \cdots + k_c$. Analogously, we get that $G$ is a Gorenstein polytope of codegree $k_{c+1} + \cdots + k_l$. Hence, we get $\codeg(F) + \codeg(G) = k_1 + \cdots + k_l = r$, again by Proposition~\ref{prop:cayley}. Now, the desired statement follows from Definition~\ref{next1}.
\end{proof}

Let us also note the following reduction that can be shown precisely as in the proof of Proposition~\ref{implication}.

\begin{lem} 
It suffices to prove Theorem~\ref{BJ-conj} for irreducible Gorenstein polytopes.
\label{lem:irr}
\end{lem}

Now, we can give the proof of the Batyrev-Juny conjecture.

\begin{proof}[Proofs of Theorem~\ref{BJ-conj}]

Let $d \ge 3s$, and $s > 0$. By Lemma~\ref{lem:irr}, we may assume that $P$ is an irreducible Gorenstein polytope. By Theorem~\ref{thm:cayley}, there exist $P_1, \ldots, P_l \subset \R^n$ such that $P$ is isomorphic to $P_1 * \cdots * P_l$ with $n = \dim(P_1 + \cdots + P_l) \le 2s-1$ and $d+1=l+n$. As $P$ is a Gorenstein polytope of codegree $d+1-s$, by Proposition~\ref{prop:cayley} there exists $(k_1, \ldots, k_l) \in \Z_{\ge 1}^l$ with $k_1 + \cdots + k_l = d+1-s$ such that $k_1 P_1 + \cdots + k_l P_l \subset \R^n$ is reflexive, 
and $\dim(\sum_{j \not=i} P_j) < n$ if $k_i \ge 2$ for $i \in \{1, \ldots, l\}$. Let us assume that $P$ is not a lattice pyramid. In this case, any $P_1, \ldots, P_l$ has dimension at least one. Therefore, Lemma~\ref{lem:mink-bound} implies $d+1-s = k_1 + \cdots + k_l \le n+1$, hence, $d \le n +s \le 2s-1+ s = 3s-1$, a contradiction.
\end{proof}

Finally, here is the proof of the refined conjecture for IDP Gorenstein polytopes. For the proof, we need the notion of the dual Gorenstein polytope. We refer to \cite{BN08} and \cite{NS13} for its definition and basic properties.

\begin{proof}[Proof of Proposition~\ref{IDP}]

Let $P$ be IDP with $d > 2s$. The dual Gorenstein polytope $P^\times$ has also degree $s$ and codegree $d+1-s$. By \cite[Cor.~2.12]{BN08}, there exist $P_1, \ldots, P_l \subset \R^n$ such that $P^\times$ is isomorphic to $P_1 * \cdots * P_l$ with $l = d+1-s$ and $d=n+l-1$. Hence, $n=d+1-l=s$. Proposition~\ref{cayley-old} implies that $P_1 + \cdots + P_l$ is reflexive. Let us assume that $P$ is irreducible, hence, also $P^\times$ is irreducible (\cite[Def.~4.13]{NS13}). In particular, $\dim(P_i) \ge 1$ for $i=1, \ldots, l$. In this case, Lemma~\ref{lem:mink-bound} (with $k_1=\cdots=k_l=1$) implies $l \le s+1$, thus, $d = l+s-1\le 2s$, a contradiction.
\end{proof}

\subsection*{Acknowledgments}

The author is indebted to Christopher Borger for a crucial conversation regarding Proposition~\ref{prop:cayley}. BN has been PI in the Research Training Group Mathematical Complexity Reduction funded by the Deutsche Forschungsgemeinschaft (DFG, German Research Foundation) - 314838170, GRK 2297 MathCoRe.

\bibliographystyle{acm}
\bibliography{bibliography}

\end{document}